\documentclass[preprint]{elsarticle}

\usepackage[utf8x]{inputenc}
\usepackage[section]{algorithm}
\usepackage{algpseudocode}
\usepackage{enumitem}
\usepackage{amsfonts}
\usepackage{amsmath}
\usepackage{amssymb}
\usepackage{amsthm}
\usepackage{verbatim}
\usepackage{graphicx}
\usepackage{float}
\usepackage{caption}
\usepackage{subcaption}
\usepackage[bookmarks]{hyperref}
\usepackage{soul}
\usepackage{comment}

\newtheorem{definition}{Definition}
\newtheorem{example}{Example}
\newtheorem{theorem}{Theorem}
\newtheorem{lemma}[theorem]{Lemma}
\newtheorem{corollary}[theorem]{Corollary}
\newtheorem{claim}[theorem]{Claim}

\newcommand{\change}{\textcolor{black}}

\graphicspath{ {./images/} }

\usepackage{xcolor}
\algdef{SE}[DOWHILE]{Do}{doWhile}{\algorithmicdo}[1]{\algorithmicwhile\ #1}

\title{Local Boxicity and Maximum Degree}
\author[AM]{Atrayee Majumder} 
\author[RgM]{Rogers Mathew} 
\address[AM]{Department of Computer Science and Engineering, Indian Institute of Technology, Kharagpur, West Bengal - 721302, India.\\ Email: atrayee.majumder@iitkgp.ac.in}
\address[RgM]{Department of Computer Science and Engineering, Indian Institute of Technology, Hyderabad, Telengana - 502285, India.\\ Email: rogersmathew@gmail.com}
\date{}

\DeclareMathOperator{\boxicity}{box}
\DeclareMathOperator{\lbox}{lbox}
\DeclareMathOperator{\posetdim}{dim}
\DeclareMathOperator{\ldim}{ldim}
\DeclareMathOperator{\proddim}{pdim}
\DeclareMathOperator{\cubicity}{cub}
\DeclareMathOperator{\minimum}{min}
\DeclareMathOperator{\maximum}{max}

\begin{document}

\begin{abstract}
The \emph{local boxicity} of a graph $G$, denoted by $\lbox(G)$, is the minimum positive integer $l$ such that $G$ can be obtained using the intersection of $k$ (, where $k \geq l$,) interval graphs where each vertex of $G$ appears as a non-universal vertex in at most $l$ of these interval graphs. Let $G$ be a graph on $n$ vertices having $m$ edges. Let $\Delta$ denote the maximum degree of a vertex in $G$. We show that,
\begin{itemize}
\item \change{$\lbox(G) \leq 2^{13\log^{*}{\Delta}} \Delta$.} There exist graphs of maximum degree $\Delta$ having a local boxicity of $\Omega(\frac{\Delta}{\log\Delta})$. 
\item $\lbox(G) \in O(\frac{n}{\log{n}})$. There exist graphs on $n$ vertices having a local boxicity of $\Omega(\frac{n}{\log n})$. 
\item \change{$\lbox(G) \leq (2^{13\log^{*}{\sqrt{m}}} + 2 )\sqrt{m}$.} There exist graphs with $m$ edges having a local boxicity of $\Omega(\frac{\sqrt{m}}{\log m})$. 
\item the local boxicity of $G$ is at most its \emph{product dimension}. This connection helps us in showing that the local boxicity of the \emph{Kneser graph} $K(n,k)$ is at most $\frac{k}{2} \log{\log{n}}$. 
\end{itemize}
The above results can be extended to the \emph{local dimension} of a partially ordered set due to the known connection between local boxicity and local dimension. Finally, we show that the \emph{cubicity} of a graph on $n$ vertices of girth greater than $g+1$ is $O(n^{\frac{1}{\lfloor g/2\rfloor}}\log n)$. 
\end{abstract}
\maketitle
\section{Introduction}
A $k$-dimensional
box or a \emph{$k$-box} is defined as the Cartesian product of closed intervals $[a_1, b_1]×[a_2, b_2]\times \cdots \times [a_k, b_k]$. A \emph{$k$-box representation} of a graph $G$ is a mapping of the vertices of $G$ to $k$-boxes in the $k$-dimensional Euclidean space such that two vertices in $G$ are adjacent if and only if their corresponding $k$-boxes have a non-empty intersection. 
\begin{definition}[Boxicity of a Graph]
The boxicity of a graph $G$, denoted by $\boxicity(G)$, is the minimum positive integer $k$ such that $G$ has a $k$-box representation. 
\end{definition}
A graph is an \emph{interval graph} if it has a $1$-box representation. An \emph{interval representation} or $1$-box representation of an interval graph $I$ is a mapping $f$ of each vertex in $I$ to a closed interval on the real line in such a way that $\forall u,v\in V(I), uv \in E(I)$ if and only if  $f(u)$ intersects $f(v)$. It is known that an interval graph may have multiple interval representations. Let $G=(V,E)$ be any graph and $G_i$, $1 \leq i \leq k$ be graphs on the same vertex set as $G$ such that $E(G) = E(G_1) \cap E(G_2) \cap \cdots \cap E(G_k)$. Then $G$ is the \emph{intersection graph of $G_i$s} and denoted by the equation $G = \cap_{i=1}^k G_i$.
Projecting the $k$-boxes in a $k$-box representation of a graph into various axes  yields an alternate definition of the boxicity of a graph which can be stated in terms of intersection of interval graphs. 
\begin{definition}[Alternate definition of boxicity, Roberts \cite{roberts1969boxicity}]
The boxicity of a graph $G$ is the minimum positive integer $k$ such that $G$ is the intersection graph of $k$ interval graphs.
\end{definition} 
It follows from the above definition of boxicity that if $G = \cap_{i=1}^k G_i$, for some graphs $G_i$, then $\boxicity(G) \leq \sum \limits_{i=1}^k \boxicity(G_i)$. The notion of boxicity was first introduced by Roberts \cite{roberts1969boxicity} in the year 1969. Cozzens \cite{cozzens1982higher} showed that computing the boxicity of a graph is NP-hard. Kratochvil \cite{kratochvil1994special} showed that determining whether the boxicity of a graph is at most two is NP-complete. See \cite{chandran2010geometric, doi:10.1137/100786290, SUNILCHANDRAN2008443, 2018arXiv180403271S} for more results on the boxicity of a graph. 
\subsection{Local boxicity of a graph}
\label{subsec:local_box}
This paper revolves around a newly introduced graph parameter, \emph{local boxicity}, which is a variant of boxicity. The notion of local boxicity was first introduced by Bl{\"a}sius, Stumpf, and Ueckerdt \cite{blasius2018local}. An \emph{l-dimensional local box} or an \emph{l-local box} is defined as the cartesian product of intervals, $\lambda_1 \times \lambda_2 \times \cdots \times \lambda_k$, $k \geq l$, where at least $k-l$ of these intervals are equal to the entire real line $\mathbb{R}$.
\begin{definition}[l-local box representation] 
A l-local box representation of a graph $G$ is a mapping of the vertices of $G$ to $l$-local boxes in the $k$-dimensional Euclidean space, $l \leq k$, such that two vertices of $G$ are adjacent in $G$ if and only if their corresponding $l$-local boxes have a non-empty intersection.
\end{definition}
\begin{definition}[Local boxicity of a graph]
The local boxicity of a graph G, denoted by $\lbox(G)$, is the minimum positive integer $l$ such that G has an l-local box representation.
\end{definition}
Projecting the $l$-local boxes in an $l$-local box representation of a graph into various axes  yields the following alternate definition of the local boxicity of a graph. 
\begin{definition}[Alternate definition of local boxicity]
Local boxicity of a graph $G$, denoted by $\lbox(G)$, is the smallest positive integer $l$ such that $G = \cap_{i=1}^k I_i$, where each $I_i$ is an interval graph and each vertex of $G$ appears as a non-universal vertex in at most $l$ interval graphs in the collection $\{I_1,I_2,\ldots, I_k\}$.
\end{definition}
It follows from their definitions that $\lbox(G) \leq \boxicity(G)$. The notion of local boxicity is useful in space efficient representation (or storage) of dense graphs having small local boxicity. As an example, consider Roberts' graph (Roberts' graph $R_n$ is the graph obtained by removing a perfect matching from a complete graph on $2n$ vertices) having boxicity $n$. Representing this graph by storing every interval graph whose intersection is the original graph requires $\Omega(n^2)$ space. Any conventional representation of the Roberts' graph using an adjacency matrix or adjacency list also requires $\Omega(n^2)$ space. However, since the local boxicity of this graph is $1$, there is a way to represent it in $O(n \log{n})$ space. A graph $G$ having local boxicity $l$ can be represented or stored in the following manner. For each vertex $v$ in the graph, an array $B_v[l][3]$ of size $l \times 3$ is maintained whose entries are as follows. For $1 \leq i \leq l$, $B_v[i][1] = $ identifier of the $i^{th}$ interval graph that contains v as a non-universal vertex, $B_v[i][2] = $ left endpoint of the interval corresponding to v in the $i^{th}$ interval graph where v is present as a non-universal vertex, and $B_v[i][3] = $ right endpoint of the interval corresponding to v in the $i^{th}$ interval graph where v is present as a non-universal vertex. Each of the entries of this array requires $O(\log{n})$ bits, where $n$ is the number of vertices in $G$. Hence, the total space required to represent $G$ is in $O(nl\log{n})$. Note that if $l$ is a constant (i.e. local boxicity of $G$ is constant) then we get a space efficient representation in $O(n\log{n})$ space.

\subsection{Dimension and local dimension of a poset}
Here we introduce the reader to the notions of the dimension and the local dimension of a partially ordered set and then in Section \ref{sec:local_box_local_dim} show its relation with the graph parameters boxicity and local boxicity introduced above. A partially ordered set or \emph{poset} $\mathcal{P}=(X,\preceq)$ is a tuple, where $X$ denotes a finite or infinite set, and $\preceq$ is a binary relation on the elements of $X$. The binary relation $\preceq$ is reflexive, anti-symmetric and transitive.
For any two elements $x,y \in X,\ x$ is said to be \emph{comparable} with $y$ if either $x \preceq y$ or $y \preceq x$. If two elements $x,y \in X$ are comparable then either $(x, y)$ is an ordered pair in $\mathcal{P}$ if $x \preceq y$ or $(y, x)$ is an ordered pair in $\mathcal{P}$ if $y \preceq x$. Otherwise, if neither $a \preceq b$ nor $b \preceq a$,
then $a$ and $b$ are called \emph{incomparable elements} of $\mathcal{P}$. In this paper, we only deal with finite posets. A \emph{linear order} is a partial order where every two elements are comparable with each other. A linear order is also called a \emph{chain} in the literature. If a partial order $\mathcal{P}=(X,\preceq)$ and a linear order $L=(X,\preceq)$ are both defined on the same set $X$, and if every ordered pair in $\mathcal{P}$ is also present in $L$, then $L$ is called a \emph{linear extension} of $\mathcal{P}$.
A collection of linear orders, say $\mathcal{L}=\{L_1,\ L_2,\ \ldots,\ L_k\}$ with each $L_i$ defined on $X$, is said to \emph{realize} a poset $\mathcal{P}=(X,\preceq)$ if, for every two distinct elements $x,y \in X$, $x \preceq y \in \mathcal{P}$ if and only if $x \preceq_{L_i} y,\ \forall L_i \in \mathcal{L}$, where $x \preceq_{L_i} y$ denotes that $x \preceq y$ is in $L_i$. We call $\mathcal{L}$ a \emph{realizer} for $\mathcal{P}$. The \emph{dimension of a poset} $\mathcal{P}$, denoted by $\posetdim(\mathcal{P})$, is defined as the minimum cardinality of a realizer for $\mathcal{P}$. The concept of poset dimension was first introduced by Dushnik and Miller \cite{dushnik1941partially} and was extensively studied by researchers since then \cite{trotter1995partially, erdos1991dimension, furedi1986dimensions, 2018arXiv180403271S}.

The notion of \emph{local dimension} was introduced recently by Ueckerdt \cite{newdim} at the \emph{Order and Geometry Workshop, 2016}. Local dimension of a poset is a variation of the standard poset dimension. The definition of local dimension originates from the concepts studied by Bl{\"a}sius, Peter Stumpf and Torsten Ueckerdt \cite{blasius2018local}, and Knauer and Ueckerdt\cite{KNAUER2016745}. A \emph{partial linear extension or ple} of a poset $\mathcal{P}$ is defined as a linear extension of any subposet of $\mathcal{P}$. 
\begin{definition}[Local Realizer]
A local realizer of a poset $\mathcal{P}$ is a family $\mathcal{L}=\{L_1,L_2,\ldots,L_l\}$ of ple's of $\mathcal{P}$ such that following conditions hold.
\begin{enumerate}
    \item If $x \preceq y$ in $\mathcal{P}$ then there exists at least one ple $L_i \in \mathcal{L}$ such that $x \preceq_{L_i} y$.
    \item If $x$ and $y$ are two incomparable elements of the poset $\mathcal{P}$, then there exist ple's $L_i, L_j \in \mathcal{L}$ such that $x \preceq_{L_i} y$ and $y \preceq_{L_j} x$.
\end{enumerate}
\end{definition}
Given a local realizer $\mathcal{L}$ of $\mathcal{P}$ and an element $x \in \mathcal{P}$, the frequency of $x$ in $\mathcal{L}$, denoted by $\mu_x(\mathcal{L})$, is defined as the number of ple's in $\mathcal{L}$ that contain $x$ as an element. The maximum frequency of a local realizer is denoted by $\mu(\mathcal{L})=\maximum \limits_{x \in \mathcal{P}} \mu_x(\mathcal{L})$.
\begin{definition}[Local Dimension]
The local dimension of a poset $\mathcal{P}$, denoted by $\ldim(\mathcal{P})$, is defined as $\minimum \mu(\mathcal{L})$ where the minimum is taken over all the local realizers $\mathcal{L}$ of $\mathcal{P}$.  
\end{definition}
For a poset $\mathcal{P}$, it follows from their definitions that   $\ldim(\mathcal{P}) \leq \posetdim(\mathcal{P})$. See \cite{TROTTER20171047, kim2018difference, barrera2019comparing, bosek2017local} for more on the local dimension of a poset. 
The notion of local dimension can be useful in space efficient representation (or storage) of dense posets having small local dimension. For example, consider the dense crown poset $S_n$ ($S_n$ is a height $2$ poset with $n$ maximal elements $b_1,b_2, \ldots, b_n$, $n$ minimal elements $a_1,a_2, \ldots, a_n$, and $a_i \preceq b_j$ for $i \neq j$. $S_n$ is considered as a standard example in the literature of poset dimension.) having $\posetdim(S_n) = n$. Representing such a poset by either storing every relation in the partial order or by storing a realizer requires $\Omega(n^2)$ space. However, since $\ldim(S_n) = 3$, there is a way to represent $S_n$ in $O(n \log{n})$ space. A poset $\mathcal{P}$ having local dimension, $\ldim(\mathcal{P})=p$ can be represented in following manner. For each element $x$ in the ground set of $\mathcal{P}$ a $2$-D array $A_x[p][2]$ of size $p \times 2$ is maintained whose entries are the following. For $1 \leq i \leq p$, $A_x[i][1] =$ identifier of the $i^{th}$ ple that contains $x$, and $A_x[i][2] =$ position of $x$ in $i^{th}$ ple. Each of the entries of this array requires $O(\log{n})$ bits to store the information, where $n$ is the number of elements in $\mathcal{P}$. Therefore, the total space required to represent $\mathcal{P}$ is in $O(np\log{n})$. Note that if $p$ is a constant (i.e. local dimension of $\mathcal{P}$ is constant) then we get a space efficient representation in $O(n\log{n})$ space. 
\subsection{Local boxicity and local dimension} 
\label{sec:local_box_local_dim}
A simple undirected graph $G_{\mathcal{P}}$ is the underlying comparability graph of a poset $\mathcal{P} = (X, \preceq)$ if $X$ is the vertex set of $G_{\mathcal{P}}$ and two vertices are adjacent in $G_{\mathcal{P}}$ if and only if they are comparable in $\mathcal{P}$. 
Let $\mathcal{P}$ be a poset and $G_{\mathcal{P}}$ be the underlying comparability graph of $\mathcal{P}$. Adiga, Bhowmick, and Chandran \cite{doi:10.1137/100786290} proved the following theorem.
\begin{theorem}[Adiga, Bhowmick, and Chandran \cite{doi:10.1137/100786290}]
\label{adigathm}
Let $\chi(G_{\mathcal{P}})$ be the chromatic number of $G_{\mathcal{P}}$ and $\chi(G_{\mathcal{P}}) \neq 1$. Then, $\frac{\boxicity(G_{\mathcal{P}})}{\chi(G_{\mathcal{P}})-1} \leq \posetdim(\mathcal{P}) \leq 2\ \boxicity(G_{\mathcal{P}})$.
\end{theorem}
A similar result connecting $\lbox(G_{\mathcal{P}})$ and $\ldim(\mathcal{P})$ was shown by Ragheb  \cite{ragheb2017local}.
\begin{lemma}[Ragheb \cite{ragheb2017local}]
\label{lemmaragheb}
$$\frac{\lbox(G_{\mathcal{P}})}{\chi(G_{\mathcal{P}})} \leq \ldim(\mathcal{P}) \leq 2\ \lbox(G_{\mathcal{P}}) + 1,$$
when $\chi(G_{\mathcal{P}}) \neq 1$.
\end{lemma} 
\subsection{Our contribution}
We know that the local boxicity of a graph is at most its boxicity. Thus all the known upper bounds for the boxicity of a graph also hold for its local boxicity. In this manuscript, we show some improved upper bounds for local boxicity. We prove the following results about the local boxicity of a graph. 
\begin{itemize}
\item Finding an upper bound for the boxicity of a graph solely in terms of its maximum degree has been extensively studied \cite{SUNILCHANDRAN2008443, esperet2009boxicity, doi:10.1137/100786290}. Let $\boxicity(\Delta)$ (or $\lbox(\Delta)$) denote the maximum boxicity (or local boxicity) of a graph having maximum degree $\Delta$.   Very recently, Scott and Wood \cite{2018arXiv180403271S} showed that $\boxicity(\Delta) = O(\Delta \log^{1+o(1)}\Delta)$. It is known due to Erd{\H{o}}s,  Kierstead, and Trotter \cite{erdos1991dimension} and Adiga, Bhowmick and Chandran \cite{doi:10.1137/100786290} that $\boxicity(\Delta) = \Omega(\Delta\log\Delta)$. \change{As for local boxicity, we show in Section \ref{subsec:max_degree_general_upperbound} that $\lbox(\Delta) \leq 2^{13\log^{*}{\Delta}} \Delta$.} Further, with the help of the result due to Kim et al. \cite{kim2018difference} on local dimension, we show that $\lbox(\Delta) = \Omega(\frac{\Delta}{\log\Delta})$.\footnote{Very recently, Esperet and Lichev \cite{esperet2020local} showed that $\lbox(\Delta) \in \Theta(\Delta)$.  } 
\item Bounding boxicity of \emph{line graphs} in terms of its maximum degree has been extensively studied in the literature. The results of Alon et al. \cite{alon2015separation} and Scott and Wood  \cite{scott2018separation} imply that the maximum local boxicity of a line graph of maximum degree $\Delta$ is $\Theta(\Delta)$. We know that line graphs belong to the class of \emph{claw-free graphs}. Using an easy constructive proof, we show in Section \emph{\ref{subsec:claw-free}} that  the local boxicity of a claw-free graph is at most $2\Delta$. We have an algorithm that gives a $2\Delta$-local box representation in $O(n\Delta^2)$ time, where $n$ is the number of vertices of the claw-free graph under consideration.
\item Esperet \cite{esperet2016boxicity} showed that every graph on $m$ edges has boxicity $O(\sqrt{m \log{m}})$. Further, in the same paper it was shown that this bound is asymptotically tight. In Section \emph{\ref{sec:size_of_a_graph}}, \change{we show that the local boxicity of a graph is  at most $(2^{13\log^{*}{\sqrt{m}}} + 2 )\sqrt{m}$.} We show the existence of graphs with $m$ edges having a local boxicity of $\Omega(\frac{\sqrt{m}}{\log m})$.
\item Roberts \cite{roberts1969boxicity} showed the the maximum boxicity of a graph on $n$ vertices is $\Theta(n)$.  In Section \ref{sec:order_of_a_graph}, we show that the maximum local boxicity of a graph on $n$ vertices is $\Theta(\frac{n}{\log{n}})$. 
\item In Section \emph{\ref{sec:prod_dimension}}, we show that the local boxicity of a graph is bounded from above by its \emph{product dimension}. This connection helps us in showing that the local boxicity of the \emph{Kneser graph} $K(n,k)$ is at most $\frac{k}{2}\log{\log{n}}$. 
\end{itemize}
The table below summarizes our results on the local boxicity of a graph listed above. With the help of Lemma~\ref{lemmaragheb} that relates the parameters local dimension and local boxicity, every upper bound that we establish for local boxicity can be extended to local dimension. 
\begin{center}
\resizebox{\linewidth}{!}{
\begin{tabular}{|c|c|c|}
\hline
\textbf{Maximum boxicity of $G$} & \textbf{Maximum local boxicity of $G$} & \textbf{Remarks}\\
\hline
$O(\Delta \log^{1+o(1)}\Delta)$, see \cite{2018arXiv180403271S} &  $\change{ 2^{13\log^{*}{\Delta}} \Delta}$, see Corollary \ref{lboxcor} & $\Delta$ is the max degree of $G$ 
\\
\hline
$\Omega(\Delta \log{\Delta})$, see \cite{erdos1991dimension} & $\Omega(\frac{\Delta}{\log \Delta})$, see Corollary \ref{lboxcor} & $\Delta$ is the max degree of $G$ \\
\hline
$O(\sqrt{m \log{m}})$, see \cite{esperet2016boxicity} &  $\change{(2^{13\log^{*}{\sqrt{m}}} + 2 )\sqrt{m}}$, see Corollary \ref{cor_lbox_edges} & $m$ denotes the no. of edges in $G$ \\
\hline
$\Omega(\sqrt{m \log m})$, see \cite{esperet2016boxicity} &  $\Omega(\frac{\sqrt{m}}{\log m})$, see Corollary \ref{cor_lbox_edges} & $m$ denotes the no. of edges in $G$ \\
\hline 
$\Theta(n)$, see \cite{roberts1969boxicity} &  $\Theta(\frac{n}{\log{n}})$, see Corollary \ref{cor_lbox_vertices} & $n$ denotes the no. of vertices in $G$ \\
\hline
--& $\leq d$, see Theorem \ref{lboxProdDim} & $d$ is the product dimension of $G$ \\
\hline
--& $\leq \frac{k}{2} \log \log n$, see Corollary \ref{lboxProdDimKneser} & when $G$ is the Kneser graph $K(n,k)$ \\
\hline
\end{tabular}
}
\end{center}
Finally, in Section \emph{\ref{Sec:girth}}, we study \emph{cubicity} (a parameter about cube representation of graphs. See Section \ref{Sec:girth} for the definition of cubicity.) of graphs of high girth. The boxicity of a graph is known to be at most its cubicity. We show that the cubicity of a graph on $n
$ vertices  having girth greater than $g+1$ is $O(n^{\frac{1}{\lfloor g/2 \rfloor}}\log n)$.
\subsection{Notations used}
Unless mentioned explicitly, all logarithms used in the paper are to the base $2$. For any positive integer $n$, we use $[n]$ to denote $\{1, \ldots , n\}$. Given a graph $G$, we shall use $V(G), E(G),$ and $\Delta(G)$ to denote its vertex set, edge set, and  maximum degree, respectively. For any $v \in V(G)$, we use $N_G(v)$ to denote the neighborhood of $v$, i.e., $N_G(v) = \{u \in V(G)~|~vu \in E(G)\}$. 
Given a graph $G$, we use $G[S]$ to denote the subgraph induced on $G$ by the vertex set $S$ for any $S \subseteq V(G)$. Similarly, we use $G[S \cup T]$ to denote the subgraph induced on $G$ by the vertex set $S \cup T$ for any $S,T \subseteq V(G)$. Let $G[S,T]$ denote the bipartite subgraph of the graph $G$ with vertex set, $V(G[S,T]) = (S \cup T)$ and edge set, $E(G[S,T]) = E(G[S \cup T])\setminus ( E(G[S]) \cup E(G[T]) )$. 

\section{Local boxicity, local dimension, and maximum degree}
In this section, we discuss the connection between the local boxicity of  a graph and its maximum degree. We begin by proving the following generalized lemma for computing the local boxicity of a graph whose vertex set is partitioned into disjoint parts.
\change{
\begin{lemma}
\label{GenLemma}
Consider a graph $G$ whose vertex set is partitioned into $k$ parts, namely $V_1, V_2, \ldots, V_k$. Let $\underset{1 \leq i < j \leq k}{\maximum}(\lbox(G[V_i \cup V_j])) = s$. Then, $\lbox(G) \leq (k-1)s$.
\end{lemma}
\begin{proof}
Let $\mathcal{I}_{i \cup j}$ denote a collection of interval graphs that corresponds to a $s$-local box representation of $G[V_i \cup V_j]$ where the vertices $v \notin V_i,V_j$ are universal in all the interval graphs in the collection. That is,
$G[V_i \cup V_j] = \bigcap \limits_{I \in \mathcal{I}_{i \cup j}} I$. It is easy to see that $G$ can be represented as follows. $$G =  \bigcap \limits_{1 \leq i < j \leq k}(\bigcap \limits_{I \in \mathcal{I}_{i \cup j}} I).$$ In such a representation, every vertex $v \in V_i$ appears as a universal vertex in every $I \in \mathcal{I}_{a \cup b}$, where $i \notin \{a,b\}$. Therefore, the maximum number of times a vertex $v \in V(G)$ is present as a non-universal vertex in the collection $\mathcal{I}_{i \cup j}$, where $1 \leq i < j \leq k$, is $(k-1)s$. Hence, the local boxicity of $G$, $\lbox(G) \leq (k-1)s$.
\end{proof}
}
\subsection{Upper bound in terms of maximum degree}
\label{subsec:max_degree_general_upperbound}
\change{
The following partitioning lemma is due to Alon et al. \cite{scott2018separation}. 
\begin{lemma}
(Lemma 3 in \cite{alon2015separation})
\label{alon_lemma}
For a graph $G$ with maximum degree $\Delta \geq 2^{64}$, there exists a partition of $V(G)$ into $r$ parts, where $r=\lceil\frac{400\Delta}{\log{\Delta}}\rceil$, such that for every vertex $v \in V(G)$ and for every part $V_i, i \in [r]$,$|N_G(v) \cap V_i| \leq \frac{1}{2}\log{\Delta}$.
\end{lemma}       
We define $\lbox(\Delta):=max\{\lbox(G):maximum\ degree\ of\ G\ is\ \Delta\}$. 
\begin{theorem}
\label{lbox_maxdeg}
For every positive integer $\Delta$, $\lbox(\Delta) \leq 2^{13\log^{*}{\Delta}} \Delta$.
\end{theorem}
\begin{proof}
It is easy to see that the local boxicity of a graph with maximum degree $1$ is at most $1$. Assume $\Delta > 1$. For such graphs, we prove by induction on $\Delta$ that their boxicity is at most $2^{13\log^{*}{\Delta}} \Delta$.
\\
\emph{Base Case: $1 < \Delta \leq 2^{64}$.} 
Esperet \cite{esperet2009boxicity} showed that for a graph $G$ with maximum degree $\Delta$,
 $box(G) \leq \Delta^2 + 2$. 
When $1 < \Delta \leq 2^{64}$, one can verify that $\Delta \leq 2^{13\log^{*}{\Delta} - 1}$. We thus have $\Delta^2 + 2 < 2\Delta^2 \leq 2\Delta \cdot     2^{13\log^{*}{\Delta} - 1} \leq 2^{13\log^{*}{\Delta}} \Delta$. 
\\
\emph{Induction step:} Let $\Delta > 2^{64}$ and assume the theorem is true for every graph with maximum degree less than $\Delta$. Let $G$ be a graph with maximum degree $\Delta$. We partition $V(G)$ into $r$ parts, namely $V_1,V_2,\ldots,V_r$, where $r = \lceil \frac{400\Delta}{\log{\Delta}} \rceil$ and $|N_G(v) \cap V_i| \leq \frac{1}{2}\log{\Delta}$, $\forall v \in V (G), i \in [r]$. Existence of such a partition is guaranteed by \emph{Lemma \ref{alon_lemma}}. For any $i,j \in [r]$, since the maximum degree of $G[V_i \cup V_j]$ is at most $\log{\Delta}$, $\lbox(G[V_i \cup V_j]) \leq \lbox(\log{\Delta})$. Applying Lemma \ref{GenLemma}, we get
    \begin{align*}
        \lbox(\Delta) & \leq (r-1)\ \lbox(\log{\Delta}) \\
        & < \left\lceil \frac{400\Delta}{\log{\Delta}} \right\rceil\ \lbox(\log{\Delta})  \\
        & \leq 2^{13} \ \frac{\Delta}{\log{\Delta}} \ \lbox(\log{\Delta})  \\
        & \leq 2^{13} \ \frac{\Delta}{\log{\Delta}} \ 2^{13\log^{*}({\log{\Delta}})} \ \log{\Delta}  \\
        & =2^{13\log^{*}{\Delta}} \Delta.
    \end{align*}
\end{proof}
\begin{theorem}[Theorem 2, Kim et al. \cite{kim2018difference}]
\label{theroem2}
The maximum local dimension of a poset on $n$ points is $\Theta(n/\log n)$.
\end{theorem}
Let $\ldim(\Delta):= \maximum \{\ldim(\mathcal{P}):$ maximum degree of $G_{\mathcal{P}}$ is $\Delta\}$ where $G_{\mathcal{P}}$ is the underlying comparability graph of a poset $\mathcal{P}$. As the maximum degree of $G_{\mathcal{P}}$. Corollary~\ref{ldimcor} follows directly from Lemma~\ref{lemmaragheb}, Theorem~\ref{lbox_maxdeg}, and Theorem~\ref{theroem2}.
\begin{corollary}
\label{ldimcor}
$\ldim(\Delta) \in \Omega(\frac{\Delta}{\log \Delta})$. Further, $\ldim(\Delta) \leq2^{1+13\log^{*}{\Delta}} \Delta+1$. 
\end{corollary}
Corollary~\ref{lboxcor} follows directly from Lemma~\ref{lemmaragheb}, Theorem~\ref{lbox_maxdeg}, and Corollary~\ref{ldimcor}.
\begin{corollary}
\label{lboxcor}
$\lbox(\Delta) \in \Omega(\frac{\Delta}{\log \Delta})$. Further, $\lbox(\Delta) \leq 2^{13\log^{*}{\Delta}} \Delta$. 
\end{corollary}
}
Bridging the gap between the upper and the lower bound for $\lbox(\Delta)$ given in Corollary~\ref{lboxcor} is certainly an interesting question. Very recently, Esperet and Lichev \cite{esperet2020local} have shown that $\lbox(\Delta) \in \Theta(\Delta)$.  

\subsection{Local boxicity and the size of a graph}
\label{sec:size_of_a_graph}
Esperet \cite{esperet2016boxicity} showed that every graph on $m$ edges has boxicity $O(\sqrt{m \log{m}})$. Further, in the same paper it is shown that this bound is asymptotically tight. In this section we explore how the local boxicity of a graph is connected with its size. 
\change{
\begin{corollary}
\label{cor_lbox_edges}
For a graph $G$ having $m$ edges, $\lbox(G) \leq (2^{13\log^{*}{\sqrt{m}}} + 2 )\sqrt{m}$. Further, there exists a graph with $m$ edges whose local boxicity is $\Omega(\frac{\sqrt{m}}{\log m})$. 
\end{corollary}
}
\begin{proof}
\change{
Let $V'$ denote the set of vertices having degree at least $\sqrt{m}$ in $G$. We have, $|V'| \leq 
\frac{2m}{\sqrt{m}} = 2\sqrt{m}$. Each vertex in $G[V(G)\setminus V']$ has degree at most 
$\sqrt{m}$ in $G$. From Corollary \ref{lboxcor} we have, $\lbox(G[V(G)\setminus V']) \leq 
2^{13\log^{*}{\sqrt{m}}} \sqrt{m}$. Therefore, $\lbox(G) \leq \lbox(G[V(G)\setminus V']) +
 2\sqrt{m} \leq 2^{13\log^{*}{\sqrt{m}}} \sqrt{m} + 2\sqrt{m} = (2^{13\log^{*}{\sqrt{m}}} + 2 )
 \sqrt{m}$.
}
\change{
Lemma 18 by Kim et al. \cite{kim2018difference} shows the existence of a height $2$ poset $
\mathcal{P}$, whose comparability graph has $n$ vertices and $\Omega(n^2)$ edges, having
  $\ldim(\mathcal{P}) \in \Omega(\frac{n}{\log n})$. Thus, $\ldim(\mathcal{P}) \in  
  \Omega(\frac{\sqrt{m}}{\log m})$ and, by Lemma~\ref{lemmaragheb}, $\lbox(G_{\mathcal{P}}) \in 
  \Omega(\frac{\sqrt{m}}{\log m})$.
 } 
\end{proof}
\change{
\begin{corollary}
Let $\mathcal{P}$ be a poset whose underlying comparability graph has $m$ edges. Then, $\ldim(\mathcal{P}) \leq (2^{13\log^{*}{\sqrt{m}}} + 2 )2\sqrt{m} + 1$.
\end{corollary}
}
 
\subsection{Constructing local box representations for claw-free graphs}
\label{subsec:claw-free}
Chandran, Mathew, and Sivadasan \cite{CHANDRAN20112359} showed that the boxicity of the line graph of a graph $G$ with maximum degree $\Delta$ is $O(\Delta \log{\log{\Delta}})$. The \emph{line graph} $L(G)$ of a graph $G$ is the graph with $V(L(G)) = E(G)$ and $E(L(G))=\{ef:e,f \in E(G)$, $e$ and $f$ share a common endpoint$\}$. 
Later Alon et al. \cite{alon2015separation} improved this bound to $O(2^{9\log^{*}{\Delta}}\Delta)$, where $\log^{*}{\Delta}$ denotes the iterated logarithm of $\Delta$, i.e., the number of times the log function is applied to get a result less than or equal to $1$. Scott and Wood \cite{scott2018separation} showed that  $\boxicity(L(G))$ of a graph $G$ with maximum degree $\Delta$ is at most $20 \Delta$, which is best possible upto a constant factor. Thus, boxicity of line graphs have been extensively studied. In this section, we study the local boxicity of claw-free graphs, a class of graphs which contains line graphs. 
A \emph{claw graph} is a complete bipartite graph $K_{1,3}$ with one part containing a single vertex and the other part containing three vertices. A \emph{claw-free graph} is a graph that contains no claw graph as its induced subgraph. 
We show that the local boxicity of a claw-free graph having a maximum degree of $\Delta$ is at most $2\Delta$. Our proof yields an algorithm for constructing $2\Delta$-local box representation for such graphs in $O(n \Delta^2)$ time.


\begin{theorem}
\label{lboxClaw}
Let $G$ be a claw-free graph and let $\chi(G)$ denote its chromatic number. Then, $\lbox(G) \leq 2(\chi(G)-1)$.
\end{theorem}
\begin{proof}
Based on an optimal vertex coloring of $G$, partition $V(G)$ into $\chi(G)$ color classes, namely $V_1, V_2, \ldots,$ $V_{\chi(G)}$. 
From Lemma~\ref{GenLemma} we know $\lbox(G) \leq (\chi(G)-1) \underset{1 \leq i < j \leq \chi(G)}{\maximum}(\lbox(G[V_i \cup V_j]))$. Since $G$ is claw-free, $G[V_i \cup V_j]$ has a maximum degree of at most $2$. That is, $G[V_i \cup V_j]$ is a disjoint union of paths and cycles. The local boxicity of a path is $1$ and that of a cycle is $2$. Thus, $\lbox(G) \leq 2(\chi(G) -1)$. 
\end{proof}
For a graph $G$ with maximum degree $\Delta$, since $\chi(G) \leq \Delta + 1$, we have the following corollary.
\begin{corollary}
\label{cor:claw-free}
Let $G$ be a claw-free graph of maximum degree $\Delta$. Then, $\lbox(G) \leq 2 \Delta$.
\end{corollary}
In the proof of Theorem~\ref{lboxClaw}, a $2$-local box representation for $G[V_i \cup V_j]$ can be obtained in $O(n)$ time, where $n$ is the number of vertices in $G$. Thus, the proof of Theorem~\ref{lboxClaw} yields an algorithm for constructing a $2 \Delta$-local box representation for $G$ that runs in $O(n \Delta^2)$ time. 
\section{Local boxicity and the order of a graph}
\label{sec:order_of_a_graph}
It is known that the maximum boxicity of a graph on $n$ vertices is $\Theta(n)$ (see \cite{roberts1969boxicity}). In this section we show that the maximum local boxicity of a graph on $n$ vertices is $\Theta(\frac{n}{\log{n}})$, where $n$ is the order of a graph. The following theorem that partitions the edges of a graph into complete bipartite graphs is due to Erd\H{o}s and Pyber. We use it in the proof of Theorem \ref{lboxonn}
\begin{theorem}[Theorem 1,  Erd\H{o}s and Pyber \cite{erdHos1997covering}]
\label{erdospyber}
Let $G$ be a graph on $n$ vertices. The edge set $E(G)$ can be partitioned into complete bipartite graphs such that each vertex $v \in V(G)$ is contained in most $c\cdot \frac{n}{\log{n}}$ of the bipartite subgraphs.
\end{theorem}
\begin{theorem}
\label{lboxonn}
Let $G$ be a graph on $n$ vertices. Then, $\lbox(G) \leq c\cdot \frac{n}{\log{n}}$.
\end{theorem}
\begin{proof}
Let $\overline{G}$ be the complement of the graph $G$ i.e. the vertex set $V(\overline{G}) = V(G)$ and the edge set $E(\overline{G})=$ $\{uv:uv \notin E(G)\}$. Now, $E(\overline{G})$ is partitioned into $k$ complete bipartite graphs $G_1, G_2, \ldots ,G_k$ using Theorem~\ref{erdospyber}. From each complete bipartite graph $G_i$ we construct one interval graph $I_i$ whose interval representation is denoted by $f_i$. Let $A_i$ and $B_i$ be the two constituting parts of the complete bipartite graph $G_i$. All the vertices $v \in A_i$ are assigned the interval $[1,2]$ and all the vertices $v \in B_i$ are assigned the interval $[3,4]$ in $f_i$. Any vertex that is not present in $G_i$ is assigned the entire real line as its interval in $f_i$.

Note that each vertex of $G$ appears as a non-universal vertex in at most $c\cdot \frac{n}{\log{n}}$ number of interval graphs from the set $\{I_1,I_2, \ldots, I_k\}$. We now argue that $G = \cap_{i=1}^k I_i$. 
\begin{claim}
\label{claimnonadj}
If $uv \notin E(G)$ then there exists exactly one interval graph $I_i$ where $uv \notin E(I_i)$.
\end{claim}
As $u$ and  $v$ are not adjacent in $G$, they are adjacent in $\overline{G}$. Since we have partitioned $E(\overline{G})$ into $k$ complete bipartite graphs, there exists exactly one complete bipartite graph, say $G_i$, which has, without loss of generality, $u \in A_i,v \in B_i$. Then, $u$ receives the interval $[1, 2]$ and $v$ receives the interval $[3,4]$ in $f_i$. Thus, $uv \notin E(I_i)$. This proves the claim. 

\begin{claim}
\label{claimadj}
If $uv \in E(G)$ then in all the interval graphs $I_i$, where $1 \leq i \leq k$, $uv \in E(I_i)$.
\end{claim}
Since $uv \in E(G)$, they are not adjacent in $\overline{G}$. In every complete bipartite graph $G_i$ that $u$ or $v$ is absent, it ($u$ or $v$) acts as a universal vertex in the interval graph $I_i$ constructed. Further, in the bipartite graphs $G_i$ where both $u$ and $v$ are present, they appear on the same part ($A_i$ or $B_i$) thus getting the same interval ($[1,2]$ or $[3,4]$) in $f_i$. Hence, $uv \in E(I_i), \forall 1 \leq i \leq k$. This proves the claim and thereby the theorem.  
\end{proof}
Combining Lemma \ref{lemmaragheb}, Theorem \ref{theroem2} and Theorem \ref{lboxonn}, we get the following corollary. 
\begin{corollary}
\label{cor_lbox_vertices}
The maximum local boxicity of a graph on $n$ vertices is $\Theta(\frac{n}{\log n})$.  
\end{corollary}
\section{Local boxicity and the product dimension of a graph}
The \emph{direct product}, denoted by $G \times H$, of graphs $G$ and $H$ has $V(G \times H) = V(G) \times V(H)$ and $E(G \times H) = \{(a,b)(c,d):a,c \in V(G),\ b,d \in V(H),\ ac \in E(G),$ and $bd \in E(H)\}$. 
The \emph{product dimension}, also known as \emph{Prague dimension}, of a graph $G$ is the minimum positive integer $k$ such that $G$ is an induced subgraph of the direct product of $k$ complete graphs. The parameter product dimension was introduced and studied by Ne{\v{s}}et{\v{r}}il 
and Pultr \cite{10.1007/3-540-08442-8_119}. More results in  \cite{Furedi2000,LOVASZ198047}. Researchers have tried to find relations between various dimensional parameters of a graph. Theorem~\ref{adigathm} relates the dimension of a poset and the boxicity of its comparability graph; F{\"u}redi \cite{Furedi2000} tries to explore a connection with the poset dimension to bound the product dimension of a Kneser graph; Chatterjee and Ghosh \cite{chatterjee2010ferrers} finds a relation between the \emph{Ferrers dimension} of a graph and its boxicity; Basavaraju et al. \cite{basavaraju2016separation} relates the \emph{separation dimension} of a graph with the boxicity of its line graph. Though a relation between the boxicity of a graph and its product dimension is not known to be explored yet, it may be noted that Chandran et al. \cite{CHANDRAN2015100} studied boxicity of the (Cartesian, strong, and direct) products of graphs. Theorem 14 in their paper gives a trivial upper bound of $nk$ for the boxicity of a graph with $n$ vertices having a product dimension of $k$. In this section, we show that the product dimension of a graph is an upper bound to its local boxicity. For this, we state below an alternate definition of product dimension by Lov{\'a}sz, Ne{\v{s}}et{\v{r}}il, and Pultr \cite{LOVASZ198047}.  
\label{sec:prod_dimension}
\begin{definition}[Lov{\'a}sz, Ne{\v{s}}et{\v{r}}il, and Pultr \cite{LOVASZ198047}]
\label{pdimdef}
The product dimension of a graph $G$, denoted by $\proddim(G)$, is the minimum positive integer $k$ for which there exists a function $f:V(G) \rightarrow \mathbb{N}^k$, such that $uv \in E(G)$, if and only if $f(u)$ and $f(v)$ differ in exactly $k$ coordinates. 
\end{definition}
\begin{theorem}
\label{lboxProdDim}
For any graph $G$, $\lbox(G) \leq \proddim(G)$.
\end{theorem}
\begin{proof}
Let $\proddim(G)=k$. Let $f:V(G) \rightarrow \mathbb{N}^k$ be a $k$-coordinate representation of $G$ (that satisfies the condition of Definition~\ref{pdimdef}) where $f_i(v)$ denotes the $i^{th}$ coordinate of $f(v)$. For each $i \in [k]$, let $S_i = \{f_i(v):v \in V(G)\}$. Let $V(G)=\{v_1,v_2, \ldots, v_n\}$. For each $i \in [k], j \in S_i$, we construct an interval graph $I_{i,j}$ in the following way. Let $g_{i,j}$ be an interval representation of $I_{i,j}$ ($g_{i,j}:V(I_{i,j}) \rightarrow X$, where $X$ is the set of all closed intervals on the real line). For a vertex $v_a \in \{v_1,v_2, \ldots, v_n\}$, if $f_i(v_a)=j$, then $g_{i,j}(v_a)=[a,a]$. Otherwise, $g_{i,j}(v_a) = [1,n]$. 

In order to show that $G = \bigcap_{i \in [k]}(\bigcap_{j \in S_i} I_{i,j})$, consider any $v_a, v_b \in V(G)$. If $v_av_b \in E(G)$, then $\forall i \in [k], f_i(v_a) \neq f_i(v_b)$ and therefore either $v_a$ or $v_b$ is a universal vertex in every interval graph we construct. Now, suppose $v_av_b \notin E(G)$. Then, for some $i \in [k]$, $f_i(v_a) = f_i(v_b)$ ($=j$, say) and therefore, from our construction, the intervals of $v_a$ and $v_b$ don't overlap in $g_{i,j}$.  Thus, $G = \bigcap_{i \in [k]}(\bigcap_{j \in S_i} I_{i,j})$. 

For any given $i \in [k]$, a vertex $v \in V(G)$ appears as a non-universal vertex in exactly one interval graph $I_{i,j}$ where $j=f_i(v)$. Thus, $v$ appears as a non-universal vertex in at most $k$ interval graphs in this collection. Hence, $\lbox(G) \leq k =  \proddim(G)$.
\end{proof}
\begin{corollary}
$\ldim(\mathcal{P}) \leq 2\ \proddim(G_{\mathcal{P}})+1$, where $\mathcal{P}=(X,\preceq)$ is a poset and $G_{\mathcal{P}}$ is its underlying comparability graph.
\end{corollary}
The \emph{Kneser graph} $K(n,k)$ is the graph whose vertex set is the set of all $k$-sized subsets of $[n]$ and any two such vertices are adjacent to each other if and only if the corresponding $k$-sized subsets do not intersect with each other. The following result is due to Gargano, K{\"o}rner, and Vaccaro \cite{gargano1994capacities} and Poljak, Pultr, and R{\"o}dl \cite{poljak1978dimension}. It is stated in Section VIII of Korner and Orlitsky \cite{korner1998zero}. 
\begin{theorem}[Gargano, K{\"o}rner, and Vaccaro \cite{gargano1994capacities}, Poljak, Pultr, and R{\"o}dl \cite{poljak1978dimension}]
\label{kneser}
$$\log{\log{n}} \leq \proddim(K(n,k)) \leq \frac{k}{2}\log{\log{n}}.$$
\end{theorem}
Corollary \ref{lboxProdDimKneser} follows directly from Theorem~\ref{lboxProdDim} and Theorem~\ref{kneser}.
\begin{corollary}
\label{lboxProdDimKneser}
$$\lbox(K(n,k)) \leq \frac{k}{2}\log{\log{n}}.$$
\end{corollary}
It would be interesting to explore how tight this upper bound for the local boxicity of a Kneser graph is. It is known from \cite{LOVASZ198047} that the graph $nK_2$, that is the graph of $2n$ vertices formed by taking $n$ copies of an edge, has a product dimension of $\Omega(\log{n})$. However, the local boxicity of this graph is one. That means that one can not give a non-trivial lower bound for the local boxicity of a graph solely in terms of its product dimension. It is worth exploring whether one can show such a lower bound in terms of the product dimension and a third parameter of the graph under consideration. 
\section{Cube representation of graphs of high girth}
\label{Sec:girth}
The \emph{girth} of a graph is the length of a smallest cycle in it. Girth of an acyclic graph is assumed to be $\infty$.
A $k$-dimensional cube or a $k$-cube is defined as the Cartesian product of unit length closed intervals $[a_1,a_1+1] \times [a_2,a_2+1] \times \cdots \times [a_k,a_k+1]$. Therefore, $k$-cubes are $k$ dimensional axis-parallel cubes. A $k$-cube representation of a graph $G$ is a mapping of the vertices of $G$ to $k$-cubes in the $k$-dimensional Euclidean space such that two vertices in $G$ are adjacent if and only if their corresponding $k$-cubes have a non-empty intersection.
\begin{definition}[Cubicity of a graph]
The cubicity of a graph $G$, denoted by $\cubicity(G)$, is defined as the minimum positive integer $k$ such that $G$ has a $k$-cube representation.
\end{definition}  
A graph is a \emph{unit interval graph} if it is an interval graph and it has an interval representation where every interval is of unit length. Below we state an alternate definition of cubicity in terms of unit interval graphs. 
\begin{definition}[Alternate definition of cubicity]
The cubicity of a graph $G$, denoted by $\cubicity(G)$, is the minimum positive integer $k$ such that there exist $k$ unit interval graphs $I_1, I_2, \ldots, I_k$ with $G = \underset{i=1}{\overset{k}{\bigcap}}I_i$.
\end{definition}
It is known and follows from their definitions that, for a graph $G$, $\lbox(G) \leq \boxicity(G) \leq \cubicity(G)$. Most graphs of high boxicity (and, thereby high cubicity) we know are graphs of low girth, whether it be the Roberts' graph $R_n$ defined in Section \ref{subsec:local_box} (a complete graph on $2n$ vertices minus a perfect matching) or the random graph studied by Erd{ő}s, Kierstead, and Trotter  \cite{erdos1991dimension}. Therefore, it is a natural question to ask whether the boxicity of a graph decreases as its girth increases. It was shown by Bhowmick and Chandran \cite{bhowmick2010boxicity} that if $G$ is an \emph{asteroidal triple free} graph having girth at least $5$, then the $\boxicity(G) \leq 2$ and $\cubicity(G) \leq 2 \lceil \log_2{\psi(G)} \rceil + 4$, where $\psi(G)$ denotes the claw number of $G$. The \emph{claw number} of a graph $G$ is the number of edges in the largest star that is an induced subgraph of $G$. Esperet and Joret \cite{esperet2013boxicity} showed that for a fixed surface $\Sigma$ there exists an integer $g_{_\Sigma}$ such that every graph with girth at least $g_{_\Sigma}$ embeddable in $\Sigma$ has boxicity at most $4$. Here we give a general upper bound for the cubicity of a graph in terms of its girth and order. We show that, for a graph $G$ on $n$ vertices with girth greater than $g+1$, $\cubicity(G) \in O(n^\frac{1}{\lfloor \frac{g}{2} \rfloor }\log{n})$. We first show in Lemma~\ref{DegenGirth} that such a graph $G$ is $\Big \lceil n^\frac{1}{\lfloor \frac{g}{2} \rfloor } \Big \rceil$-degenerate and then use Theorem~\ref{CubDgn} (due to Adiga, Chandran, and Mathew \cite{adiga2014cubicity}) stated below, to obtain the result. 
\begin{definition}
A graph is $k$-degenerate if the vertices of the graph can be enumerated in such a way that every vertex is followed by at most $k$ of its neighbors. The least number $k$ such that the graph is $k$-degenerate is called the degeneracy of the graph.
\end{definition}
\begin{theorem}[Theorem 1, Adiga, Chandran, and Mathew \cite{adiga2014cubicity}]
\label{CubDgn}
For every $k$-degenerate graph $G$, $\cubicity(G) \leq (k+2)\lceil2e \log{n} \rceil$. 
\end{theorem}
\begin{lemma}
\label{DegenGirth}
Let $G$ be a graph on $n$ vertices having girth greater than $g+1$. Then $G$ is $k$-degenerate, where $k = \Big \lceil n^\frac{1}{\lfloor \frac{g}{2} \rfloor } \Big \rceil$.
\end{lemma}
\begin{proof}
We will prove this lemma by contradiction. Suppose the lemma is not true. Then, there exists a $V_1 \subseteq V(G)$ such that $G[V_1]$ is connected  and every vertex in $G[V_1]$ has degree greater than $k$. Consider a vertex $v$ in $G[V_1]$. Let $S$ be the set of vertices that are at a distance of at most $\lfloor \frac{g}{2} \rfloor$ from $v$. Since girth of $G[V_1]$ is greater than $g+1$, $G[S]$ is a tree. As the minimum degree of a vertex in $G[V_1]$ is greater than $k$, the leaf vertices of $G[S]$ have to be adjacent to some vertices in $V_1 \setminus S$. Thus, the set $V_1 \setminus S$ is non-empty. But then we have $|S| > k^{\lfloor \frac{g}{2} \rfloor} = n$, contradicting the fact that $V_1 \setminus S$ is non-empty. Hence our assumption that $G$ is not $k$-degenerate is false. 
\end{proof}
From Lemma~\ref{DegenGirth} and Theorem~\ref{CubDgn}, we have the following theorem.
\begin{theorem}
\label{CubGirth}
Let $G$ be a graph on $n$ vertices with girth greater than $g+1$. Then, $\cubicity(G)$ is in $O(n^\frac{1}{\lfloor \frac{g}{2} \rfloor }\log{n})$. More precisely, $\cubicity(G) \leq (n^\frac{1}{\lfloor \frac{g}{2} \rfloor}+2)\lceil2e \log{n} \rceil$.
\end{theorem}
\begin{example}
\label{CubStar}
It was shown by Adiga and Chandran \cite{adiga2010cubicity} that the cubicity of a graph $G$ is at least $\lceil \log_2{\psi(G)} \rceil$, where $\psi(G)$ is the claw number of $G$. Consider the star graph $S_{1,n-1}$ on $n$ vertices having a claw number of $n-1$. Since the girth of a tree is assumed to be $\infty$, Theorem~\ref{CubGirth} gives an asymptotically tight upper bound for the cubicity of $S_{1,n-1}$. 
\end{example}
\begin{example}
\label{CubPend}
Alon, Ganguly, and Srivastava  \cite{alon2019high} showed that there exists a graph $G$ on $n$ vertices having girth at least $\frac{log_5{n}}{4}$. Consider such a graph $G$. Take a vertex $v$ in $G$. We add $n$ pendent vertices to $v$ to construct a graph $G^{'}$. Thus, the claw number of $G^{'}$, $\psi(G^{'}) \geq n$. Hence, $\cubicity(G^{'}) \in \Omega(\log{n})$. Theorem~\ref{CubGirth} gives an upper bound for the cubicity of $G^{'}$. 
\end{example}
\begin{example}
\label{CubBip}
Consider the bipartite graph $G$ obtained by removing a perfect matching from a complete bipartite graph, $K_{n,n}$ on $2n$ vertices. It is known that the boxicity (and thereby cubicity) of $G$ is at least $\frac{n}{2}$. Applying Theorem~\ref{CubGirth} with $g=2$, we get $\cubicity(G) \in O(n \log{n})$.
\end{example}
From Example~\ref{CubStar} and Example~\ref{CubPend}, we observe that there are graphs of high girth on $n$ vertices whose cubicity is in $\Omega(\log{n})$. From Example~\ref{CubBip}, we observe that for a graph on $n$ vertices with girth greater than $g+1$, we cannot get a bound of $O(n^{\alpha g})$ for its cubicity, where $\alpha$ is a constant less than $\frac{1}{2}$. From these two observations, we believe it would be worthwhile to try improving the bound in Theorem~\ref{CubGirth} to $(c_1 n^{\frac{1}{\lfloor \frac{g}{2} \rfloor}} + c_2 \log{n})$, where $c_1$ and $c_2$ are constants.
\section*{Acknowledgement} The authors would like to thank the anonymous reviewers for their valuable suggestions. 
\section{References}
\bibliographystyle{plain}

\begin{thebibliography}{}

\end{thebibliography}


\begin{thebibliography}{10}

\bibitem{doi:10.1137/100786290}
Abhijin Adiga, Diptendu Bhowmick, and L.~Sunil Chandran.
\newblock Boxicity and poset dimension.
\newblock {\em SIAM Journal on Discrete Mathematics}, 25(4):1687--1698, 2011.

\bibitem{adiga2010cubicity}
Abhijin Adiga and L.~Sunil Chandran.
\newblock Cubicity of interval graphs and the claw number.
\newblock {\em Journal of Graph Theory}, 65(4):323--333, 2010.

\bibitem{adiga2014cubicity}
Abhijin Adiga, L.~Sunil Chandran, and Rogers Mathew.
\newblock Cubicity, degeneracy, and crossing number.
\newblock {\em European Journal of Combinatorics}, 35:2--12, 2014.

\bibitem{alon2015separation}
Noga Alon, Manu Basavaraju, L.~Sunil Chandran, Rogers Mathew, and Deepak
  Rajendraprasad.
\newblock Separation dimension of bounded degree graphs.
\newblock {\em SIAM Journal on Discrete Mathematics}, 29(1):59--64, 2015.

\bibitem{alon2019high}
Noga Alon, Shirshendu Ganguly, and Nikhil Srivastava.
\newblock High-girth near-ramanujan graphs with localized eigenvectors.
\newblock {\em arXiv preprint arXiv:1908.03694}, 2019.

\bibitem{barrera2019comparing}
Fidel Barrera-Cruz, Thomas Prag, Heather~C Smith, Libby Taylor, and William~T
  Trotter.
\newblock Comparing dushnik-miller dimension, boolean dimension and local
  dimension.
\newblock {\em Order}, pages 1--27, 2019.

\bibitem{basavaraju2016separation}
Manu Basavaraju, L.~Sunil Chandran, Martin~C. Golumbic, Rogers Mathew, and
  Deepak Rajendraprasad.
\newblock Separation dimension of graphs and hypergraphs.
\newblock {\em Algorithmica}, 75(1):187--204, 2016.

\bibitem{bhowmick2010boxicity}
Diptendu Bhowmick and L.~Sunil Chandran.
\newblock Boxicity and cubicity of asteroidal triple free graphs.
\newblock {\em Discrete mathematics}, 310(10-11):1536--1543, 2010.

\bibitem{blasius2018local}
Thomas Bl{\"a}sius, Peter Stumpf, and Torsten Ueckerdt.
\newblock Local and union boxicity.
\newblock {\em Discrete Mathematics}, 341(5):1307--1315, 2018.

\bibitem{bosek2017local}
Bart{\l}omiej Bosek, Jaros{\l}aw Grytczuk, and William~T Trotter.
\newblock Local dimension is unbounded for planar posets.
\newblock {\em arXiv preprint arXiv:1712.06099}, 2017.

\bibitem{SUNILCHANDRAN2008443}
L.~Sunil Chandran, Mathew~C. Francis, and Naveen Sivadasan.
\newblock Boxicity and maximum degree.
\newblock {\em Journal of Combinatorial Theory, Series B}, 98(2):443 -- 445,
  2008.

\bibitem{chandran2010geometric}
L~Sunil Chandran, Mathew~C Francis, and Naveen Sivadasan.
\newblock Geometric representation of graphs in low dimension using axis
  parallel boxes.
\newblock {\em Algorithmica}, 56(2):129--140, 2010.

\bibitem{CHANDRAN2015100}
L.~Sunil Chandran, Wilfried Imrich, Rogers Mathew, and Deepak Rajendraprasad.
\newblock Boxicity and cubicity of product graphs.
\newblock {\em European Journal of Combinatorics}, 48:100 -- 109, 2015.
\newblock Selected Papers of Eurocomb’13.

\bibitem{CHANDRAN20112359}
L.~Sunil Chandran, Rogers Mathew, and Naveen Sivadasan.
\newblock Boxicity of line graphs.
\newblock {\em Discrete Mathematics}, 311(21):2359 -- 2367, 2011.

\bibitem{chatterjee2010ferrers}
Soumyottam Chatterjee and Shamik Ghosh.
\newblock Ferrers dimension and boxicity.
\newblock {\em Discrete mathematics}, 310(17-18):2443--2447, 2010.

\bibitem{cozzens1982higher}
Margaret~B. Cozzens.
\newblock Higher and multi-dimensional analogues of interval graphs.
\newblock {\em Ph.D. thesis, Department of Mathematics, Rutgers University, New
  Brunswick, NJ}, 1982.

\bibitem{dushnik1941partially}
Ben Dushnik and Edwin~W. Miller.
\newblock Partially ordered sets.
\newblock {\em American Journal of Mathematics}, 63(3):600--610, 1941.

\bibitem{erdos1991dimension}
Paul Erd{\H{o}}s, Henry~A. Kierstead, and William~T. Trotter.
\newblock The dimension of random ordered sets.
\newblock {\em Random Structures \& Algorithms}, 2(3):253--275, 1991.

\bibitem{erdHos1997covering}
Paul Erd{\H{o}}s and L{\'a}szl{\'o} Pyber.
\newblock Covering a graph by complete bipartite graphs.
\newblock {\em Discrete mathematics}, 170(1-3):249--251, 1997.

\bibitem{esperet2009boxicity}
Louis Esperet.
\newblock Boxicity of graphs with bounded degree.
\newblock {\em European Journal of Combinatorics}, 30(5):1277--1280, 2009.

\bibitem{esperet2016boxicity}
Louis Esperet.
\newblock Boxicity and topological invariants.
\newblock {\em European Journal of Combinatorics}, 51:495--499, 2016.

\bibitem{esperet2013boxicity}
Louis Esperet and Gwena{\"e}l Joret.
\newblock Boxicity of graphs on surfaces.
\newblock {\em Graphs and combinatorics}, 29(3):417--427, 2013.

\bibitem{esperet2020local}
Louis Esperet and Lyuben Lichev.
\newblock Local boxicity.
\newblock {\em arXiv preprint arXiv:2012.04569}, 2020.

\bibitem{Furedi2000}
Zolt{\'a}n F{\"u}redi.
\newblock {\em On the Prague Dimension of Kneser Graphs}, pages 125--128.
\newblock Springer US, Boston, MA, 2000.

\bibitem{furedi1986dimensions}
Zoltan F{\"u}redi and Jeff Kahn.
\newblock On the dimensions of ordered sets of bounded degree.
\newblock {\em Order}, 3(1):15--20, 1986.

\bibitem{gargano1994capacities}
Luisa Gargano, J{\'a}nos K{\"o}rner, and Ugo Vaccaro.
\newblock Capacities: from information theory to extremal set theory.
\newblock {\em Journal of Combinatorial Theory, Series A}, 68(2):296--316,
  1994.

\bibitem{kim2018difference}
Jinha Kim, Ryan~R. Martin, Tom{\'a}{\v{s}} Masa{\v{r}}{\'\i}k, Warren Shull,
  Heather~C. Smith, Andrew Uzzell, and Zhiyu Wang.
\newblock On difference graphs and the local dimension of posets.
\newblock {\em European Journal of Combinatorics}, 86:103074, 2020.

\bibitem{KNAUER2016745}
Kolja Knauer and Torsten Ueckerdt.
\newblock Three ways to cover a graph.
\newblock {\em Discrete Mathematics}, 339(2):745 -- 758, 2016.

\bibitem{korner1998zero}
Janos Korner and Alon Orlitsky.
\newblock Zero-error information theory.
\newblock {\em IEEE Transactions on Information Theory}, 44(6):2207--2229,
  1998.

\bibitem{kratochvil1994special}
Jan Kratochv{\'\i}l.
\newblock A special planar satisfiability problem and a consequence of its
  {NP}-completeness.
\newblock {\em Discrete Applied Mathematics}, 52(3):233--252, 1994.

\bibitem{LOVASZ198047}
L{\'a}szl{\'o} Lov{\'a}sz, Jaroslav Ne{\v{s}}et{\v{r}}il, and Ale\v{s} Pultr.
\newblock On a product dimension of graphs.
\newblock {\em Journal of Combinatorial Theory, Series B}, 29(1):47 -- 67,
  1980.

\bibitem{10.1007/3-540-08442-8_119}
Jaroslav Ne{\v{s}}et{\v{r}}il and Ale\v{s} Pultr.
\newblock A dushnik - miller type dimension of graphs and its complexity.
\newblock In Marek Karpi{\'{n}}ski, editor, {\em Fundamentals of Computation
  Theory}, pages 482--493, Berlin, Heidelberg, 1977. Springer Berlin
  Heidelberg.

\bibitem{poljak1978dimension}
Svatopluk Poljak, Ales Pultr, and Vojt{\v{e}}ch R{\"o}dl.
\newblock On the dimension of kneser graphs.
\newblock {\em Algebraic methods in graph theory}, 1:631646, 1978.

\bibitem{ragheb2017local}
Rameez Ragheb.
\newblock On the local dimension of partially ordered sets.
\newblock {\em Masters thesis, Department of Mathematics, Karlsruher Institut
  f{\"u}r Technologie, Karlsruhe, Germany}, 2017.

\bibitem{roberts1969boxicity}
Fred~S. Roberts.
\newblock On the boxicity and cubicity of a graph.
\newblock {\em Recent progress in combinatorics}, 1:301--310, 1969.

\bibitem{2018arXiv180403271S}
Alex Scott and David Wood.
\newblock Better bounds for poset dimension and boxicity.
\newblock {\em Transactions of the American Mathematical Society},
  373(3):2157--2172, 2020.

\bibitem{scott2018separation}
Alex Scott and David~R Wood.
\newblock Separation dimension and degree.
\newblock In {\em Mathematical Proceedings of the Cambridge Philosophical
  Society}, pages 1--10. Cambridge University Press, 2018.

\bibitem{trotter1995partially}
William~T Trotter.
\newblock Partially ordered sets.
\newblock {\em Handbook of combinatorics}, 1:433--480, 1995.

\bibitem{TROTTER20171047}
William~T. Trotter and Bartosz Walczak.
\newblock Boolean dimension and local dimension.
\newblock {\em Electronic Notes in Discrete Mathematics}, 61:1047 -- 1053,
  2017.

\bibitem{newdim}
Torsten Ueckerdt.
\newblock {\em Order \& Geometry Workshop}, 2016.

\end{thebibliography}

\end{document}